\documentclass[12pt,reqno]{amsart}

\usepackage{amsmath,amsfonts,amssymb,amscd,amsbsy,epsf,graphicx}
\usepackage{float}

 \newtheorem{thm}{Theorem}[section]
 \newtheorem{conj}{Conjecture}[section]
 \newtheorem{cor}[thm]{Corollary}
 \newtheorem{lem}[thm]{Lemma}
 
 \theoremstyle{definition}
 
 \theoremstyle{remark}
 
 \numberwithin{equation}{section}

 \newcommand{\Z}{\mathbb{Z}}

 \newcommand{\F}{\mathbb{F}}

 \newcommand{\Sc}{\operatorname{Sc}} 
 \newcommand{\tr}{\operatorname{tr}}
 \newcommand{\bs}\boldsymbol{}
 \newcommand{\CP}{\mathcal{P}}

\newcommand{\GL}{GL}

\begin{document}

\title[Multidimensional integrals and Painlev\'e V]
{Some multidimensional integrals in number theory and connections with the Painlev\'e V equation}

\author{Estelle Basor}
\address{American Institute of Mathematics}
\email{ebasor@aimath.org}

\author{Fan Ge}
\address{Pure Mathematics, University of Waterloo}
\email{fange.math@gmail.com}

\author{Michael O. Rubinstein}
\address{Pure Mathematics, University of Waterloo}
\email{mrubinst@uwaterloo.ca}

\thanks{The third author is Supported in part by an NSERC Discovery Grant}

\subjclass{Primary 11M06, Secondary 33E17}

\keywords{Painlev\'e, Hankel determinants, divisor sums, elliptic aliquot cycles}


\begin{abstract}
We study piecewise polynomial functions $\gamma_k(c)$ that appear in the asymptotics of averages of
the divisor sum in short intervals. Specifically, we express these polynomials as the inverse Fourier
transform of a Hankel determinant that satisfies a Painlev\'e V equation.
We prove that
$\gamma_k(c)$ is very smooth at its transition points, and also
determine the asymptotics of $\gamma_k(c)$ in a large neighbourhood of $k=c/2$.
Finally, we consider the coefficients that appear in the asymptotics of
elliptic Aliquot cycles.
\end{abstract}

\maketitle

\section{Introduction}

{\bf Asymptotics of the mean square of sums of the $k$-th divisor function over short intervals.}
Let $d_k(n)$ be the $k$-th divisor numbers, i.e. the Dirichlet coefficients of the $k$-th power
of the Riemann zeta function:
\begin{eqnarray}
    \label{eq:zeta k}
    \zeta(s)^k = \sum_1^\infty \frac{d_k(n)}{n^s}, \qquad \Re{s} > 1.
\end{eqnarray}
The Dirichlet coefficient $d_k(n)$ is equal to the number of ways of writing $n$ as
a product of $k$ factors.
Define
\begin{eqnarray}
    \label{eq:Delta}
    S_k(X) = \sum_{n \leq X} d_k(n).
\end{eqnarray}

Let $XP_{k-1}(\log{X})$ be the residue, at $s=1$ of $\zeta(s)^k X^s/s$, with
$P_{k-1}(\log{X})$ being a polynomial in $\log{X}$ of degree $k-1$. Then
\begin{eqnarray}
    \label{eq:main and remainder}
    S_k(X) =  X P_{k-1}(\log{X}) + \Delta_k(X),
\end{eqnarray}
with $\Delta_k(X)$ denoting the remainder term.

The $k$ divisor problem states that the true order
of magnitude for $\Delta_k$ is:
\begin{eqnarray}
    \label{eq:divisor problem}
    \Delta_k(X) = O\left(X^{(k-1)/2k+\epsilon}\right).
\end{eqnarray}

When $k=2$, the traditional Dirichlet divisor problem is
\begin{eqnarray}
    \label{eq:k=2}
    D_2(X) = X\log{X} +(2\gamma-1)X+\Delta_2(X),
\end{eqnarray}
with a conjectured remainder
\begin{eqnarray}
    \label{eq:dirichlet divisor problem}
    \Delta_2(X) = O\left(X^{1/4+\epsilon}\right).
\end{eqnarray}

The estimate for the remainder term $\Delta_k(X)$
is based on expected cancellation in Voronoi-type formulas for $\Delta_k(X)$
and also on estimates, due to Cram\'er~\cite{C} ($k=2$) and Tong~\cite{T} ($k>2$), for the mean square of $\Delta_k$.


Let
\begin{equation}\label{Def of Delta(x:H)}
\Delta_k(x;H) = \Delta_k(x+H)-\Delta_k(x)
\end{equation}
be the remainder term for sums of $d_k$ over the interval
$[x,x+H]$.


Define
\begin{equation}\label{Lester const}
    a_k = \prod_p \Big\{ (1- \frac 1p)^{k^2}\sum_{j=0}^\infty
    \left(\frac{\Gamma(k+j)}{\Gamma(k)j!}\right)^2 \frac 1{p^j} \Big\}.
\end{equation}

Keating, Rodgers, Roditty-Gershon, and Rudnick conjectured~\cite{KRRR}:
\begin{conj}\label{conj:mean square SI}
If $0<\alpha <1-\frac 1{k}$ is fixed, then for $H=X^\alpha$,
\begin{equation}
\frac 1X\int_X^{2X} \Big(\Delta_k(x,H) \Big)^2 dx \sim  a_k \mathcal
P_k(\alpha) H(\log X)^{k^2-1}\;,\quad X\to \infty
\end{equation}
where
$P_k(\alpha)$ is given by
\begin{equation}\label{def of P_k}
\mathcal P_k(\alpha) =(1-\alpha)^{k^2-1} \gamma_k(\frac 1{1-\alpha})
\;.
\end{equation}
Here
\begin{equation} 
    \label{eq:def of gamma}
    \gamma_k(c) = \frac{1}{k!\, G(1+k)^2} \int_{[0,1]^k} \delta(t_1 +
    \ldots + t_k -c) \prod_{i< j}(t_i-t_j)^2\, dt_1\ldots dt_k,
\end{equation}
$G$ is the Barnes $G$-function, so that for
positive integers $k$, $G(1+k) = 1!\cdot 2! \cdot 3! \cdots (k-1)!$.
\end{conj}
For $1-\frac 1{k-1}<\alpha<1-\frac 1k$, the conjecture is consistent with a theorem of Lester~\cite{L}.

Let $U$ be an $N\times N$ matrix.
The {\em secular coefficients} $\Sc_j(U)$ are the coefficients  of the
characteristic polynomial of $U$:
\begin{equation}
\det(I+xU) =\sum_{j=0}^N \Sc_j(U) x^j
\end{equation}
Thus $\Sc_0(U)=1$, $\Sc_1(U) = \tr U$, $\Sc_N(U) = \det U$.  The secular coefficients are the elementary symmetric functions in
the eigenvalues of $U$.

Define the matrix integrals, with respect to Haar measure,
over the group $U(N)$ of $N\times N$ unitary matrices:
\begin{equation}\label{def of I}
I_k(m;N):=\int_{U(N)} \Big| \sum_{\substack{j_1+\dots + j_k=m\\
0\leq j_1,\dots, j_k \leq N}}\Sc_{j_1}(U)  \dots
\Sc_{j_k}(U)\Big|^2 dU\;.
\end{equation}

\begin{thm}[KR$^3$]\label{theorem:Asymp of I}
\label{integral} Let $c := m/N$. Then for $c \in [0,k]$,
\begin{equation}
\label{theorem:poly_approx} I_k(m;N) = \gamma_k(c) N^{k^2-1} +
O_k(N^{k^2-2}),
\end{equation}
with
\begin{equation}
    \gamma_k(c) = \frac{1}{k!\, G(1+k)^2} \int_{[0,1]^k} \delta(t_1 +
    \ldots + t_k-c) \prod_{i< j}(t_i-t_j)^2\, dt_1 \ldots dt_k,
\end{equation}
%
\end{thm}

KR$^3$ also proved the matrix integral satisfies a functional equation $I_k(m;N) = I_k(kN-m;N)$,
from which it follows that 
\begin{equation}
\gamma_k(c) = \gamma_k(k-c),
\end{equation}
and also that
\begin{thm}[KR$^3$]\label{theorem:PP} 
\label{integral analytic}
\begin{equation}
\label{gamma analytic}
\gamma_k(c) =\sum_{0\leq \ell <c}\binom{k}{\ell}^2 (c-\ell)^{(k-\ell)^2+\ell^2-1} g_{k,\ell} (c-\ell)
\end{equation}
where $g_{k,\ell}(c-\ell)$ are (complicated) polynomials in $c-\ell.$ 
\end{thm} 

For a fixed $k$, $\gamma_k(c)$ is a piecewise polynomial function of $c$.  Specifically,  it is a fixed polynomial for $r\leq c<r+1$ ($r$ integer), and each time the value of $c$ passes through an integer it
becomes a different polynomial.

For example,
\begin{equation}
\gamma_2(c) = \frac {1}{2!} \int_{\substack{0\leq t_1\leq
1\\0\leq c-t_1\leq 1}} (t_1-(c-t_1))^2\,dt_1 = \begin{cases}
\frac{c^3}{3!},& 0\leq c\leq 1 \\  & \\ \frac {(2-c)^3}{3!},&1\leq
c\leq 2\end{cases}
\end{equation}
and
\begin{equation}\label{eq:rod3easy}
\gamma_3(c) =
\begin{cases} \frac 1{8!}c^8,&0<c<1\\
\frac 1{8!}(3-c)^8,&2<c<3
\end{cases}
\end{equation}
while for $1<c<2$ we get
\begin{multline}\label{eq:rod3mid}
    \gamma_3(c)=\frac 1{8!}\Big(-2 c^8+24 c^7-252 c^6+1512 c^5-4830c^4 \\
    +8568 c^3-8484 c^2+4392 c-927\Big).
\end{multline}

\section{Relationship to a Hankel determinant}

Our starting point is to derive an expression for $\gamma_k(c)$ as the Fourier transform of a Hankel determinant.
In~\eqref{eq:def of gamma}, we substitute for the Dirac delta function:
\begin{equation}
    \label{eq:delta}
    \delta(x) = \int_{-\infty}^{\infty} \exp(2\pi i x y) dy.
\end{equation}
One can be rigorous by writing $\delta(x)$ as the limit of a highly peaked Gaussian, i.e.
as the inverse Fourier transform of a highly spread out Gaussian, but for convenience we proceed
as above.

Thus
\begin{multline}
    \label{eq:fourier transform 1}
    \gamma_k(c) = \frac{1}{k!\, G(1+k)^2}
    \int_{-\infty}^{\infty} \exp(2\pi i u c)
    \int_{[0,1]^k} \exp\left( -2\pi i u \sum t_j\right) \\
    \times \prod_{i< j}(t_i-t_j)^2\, dt_1 \ldots dt_k du.
\end{multline}
We also note a more symmetric form of the above by substituting $t_j=x_j+1/2$, so that
\begin{multline}
    \label{eq:fourier transform 2}
    \gamma_k(c) = \frac{1}{k!\, G(1+k)^2}
    \int_{-\infty}^{\infty} \exp(2\pi i u (c-k/2))
    \int_{[-1/2,1/2]^k} \exp\left( -2\pi i u \sum x_j\right) \\
    \times \prod_{i< j}(x_i-x_j)^2\, dt_1 \ldots dx_k du.
\end{multline}

We will prove the following two formulas for~$\gamma_k(c)$.

\begin{thm}\label{thm:two identities}
\begin{equation}
    \gamma_k(c)=
    \frac{1}{G(1+k)^2(2\pi i)^{k(k-1)}}
    \int_{-\infty}^{\infty} \exp(2\pi i u c)
    \det_{k\times k} \left( f^{(i+j-2)}(u)\right)
     du
    \label{eq:gamma as ft of det}
\end{equation}
where
$f(u)= \int_0^1 \exp(-2 \pi i u t) dt = (1-\exp(-2\pi i u))/(2 \pi i u)$.
The determinant is a Hankel determinant.

A similar, but more symmetric, identity is:
\begin{multline}
    \label{eq:gamma as ft of det b}
    \gamma_k(c) =
    \frac{1}{G(1+k)^2(2\pi i)^{k(k-1)}}
    \int_{-\infty}^{\infty} \exp(2\pi i u (c-k/2))
    \det_{k\times k} \left( h^{(i+j-2)}(u)\right)
     du
\end{multline}
where $h(u)= \int_{-1/2}^{1/2} \exp(-2 \pi i u x) dx = \sin(\pi u)/(\pi u)$.
\end{thm}

Our proof will use the Andreief identity:


\begin{lem}[Andreief]
Let $A_k(t), B_k(t),r(t)$ be integrable functions on the interval $[a,b]$.
Then
\begin{eqnarray}
    \label{eq:andreief b}
    &&\frac{1}{N!}
    \int_{[a,b]^N}
    \prod_{j=1}^N r(t_j)
    \det_{N\times N} \left( A_k(t_j)\right)
    \det_{N\times N} \left( B_k(t_j)\right)
    dt_1 \ldots
    dt_N  \\
    &&=
    \det_{N\times N}
    \left(
        \int_a^b r(t) A_j(t) B_k(t) dt
    \right).
\end{eqnarray}
\end{lem}

\begin{proof}[Proof of Theorem~\ref{thm:two identities}]
To prove the first identity in~\ref{thm:two identities}, apply Andreief's identity to equation~\eqref{eq:fourier transform 1}, with $A$ and $B$ two Vandermonde determinants,
and $r(t)=\exp(-2\pi i u t)$, to get:
\begin{multline}
    \gamma_k(c) = \\
    \frac{1}{G(1+k)^2}
    \int_{-\infty}^{\infty} \exp(2\pi i u c)
    \det_{k\times k} \left( \int_0^1 \exp(-2\pi i u t) t^{i+j-2} dt\right) 
    du
\end{multline}
The entries of the matrix can be expressed as derivatives, with respect to $u$, of
$\int_0^1 \exp(-2 \pi i u t) dt$, and we can then correct for the extra powers of $-2\pi i u$ by
dividing the $l$-th row by $(-2\pi i u)^{l-1}$ and the $j$-th column by $(-2\pi i u)^{j-1}$, thus by
$(-2\pi i u)^{k(k-1)}$ in total (and then dropping the $-1$ since $k(k-1)$ is even).

Using the second form~\eqref{eq:fourier transform 2}, we similarly have~\eqref{eq:gamma as ft of det b}
where $h(u)= \int_{-1/2}^{1/2} \exp(-2 \pi i u x) dx = \sin(\pi u)/(\pi u)$.
\end{proof}

Some of the basic properties of $\gamma_k(c)$ can be read from~\eqref{eq:gamma as ft of det}.
For example, the inverse Fourier transform of $f^{(j)}$ is equal to $(-2\pi i)^j c^j$ on the interval
$(0,1)$ and 0 outside this interval. Expanding the determinant as a permutation sum,
each summand thus has inverse Fourier transform a convolution
of such terms, and is thus supported on $c \in (0,k)$.

It also shows that $\gamma_k(c)$ is a polynomial
in $c$ on each interval $[j,j+1]$, $0 \leq j \leq k-1$ of degree at most $k^2-1$,
because the $i,j$ entry has inverse Fourier Transform a polynomial
in $c$ on $(0,1)$ of degree $i+j-2$.
Multiply out the determinant as a permutation sum. Each summand, when integrated with respect to $c$,
is the inverse Fourier transform of a product of $k$ functions, and hence
consists of $k-1$ convolutions of the individual inverse Fourier transforms. Each convolution increases the degree
of the polynomial by 1. Hence, each permutation $\sigma$ has its resulting degree bounded by
$(k-1)+\sum_{i=1}^k (i+\sigma_i-2) = k^2-1$.


We can thus use~\eqref{eq:gamma as ft of det} to compute the polynomials $\gamma_k(c)$ by
evaluating it at $\geq k^2$ rational
values of $c$, say, in each unit interval and interpolating. In this manner, we determined the polynomials
$\gamma_k(c)$ listed in Table~\ref{table:gamma_k} and~\ref{table:gamma_k b}.

In the symmetric form~\eqref{eq:gamma as ft of det b}, one also sees that $\gamma_k(c) = \gamma_k(c-k)$,
by substituting $-u$ for $u$, and using the fact
that the determinant in that formula is an even function of $u$.

Setting \begin{equation} \label{eq:g} g(t) = \int_0^1 \exp(-tx) dx, \end{equation}
so that \begin{equation} g^{(n)}(t) = \int_0^1 (-x)^n \exp(-tx) dx, \end{equation}
and letting
\begin{equation}
    D_k(t) = \det_{k \times k} ( g^{(i+j-2)}(t) ),
\end{equation}
we have that~\eqref{eq:gamma as ft of det} can be written as
\begin{equation}
    \label{eq:gamma as ft D_k}
    \gamma_k(c) = \frac{1}{G(k+1)^2} \int_{-\infty}^{\infty} \exp(2 \pi i c u)  D_k(2 \pi i u) du.
\end{equation}

$D_k(t)$ also satisfies a Painlev\'e V equation. This is proven 
in more generality in a paper of Basor, Chen and Ehrhardt~\cite{BCE} (4.38 of that paper,
with $a=0$, $b=t$, $\alpha = 0$).
Specifically, the following holds.

\begin{thm}
Let
\begin{equation}
    H_k(t)=
    t \frac{D'_k(t)}{D_k(t)}+k^2.
\end{equation}
Then
\begin{multline}
    (tH''_k(t))^2 = \\
    (H_k(t)+(2k-t)H'_k(t))^2
    -4(H'_k(t))^2(k^2-H_k(t)+tH'_k(t)).
\end{multline}
\end{thm}

Another interesting feature, is that, while $\gamma_k(c)$ is
given by a different polynomial on each $[j,j+1]$, $0 \leq j \leq k-1$,
$\gamma_k(c)$ can be differentiated $j^2+(k-j)^2-2$ times at $c=j$, i.e. is very smooth.

\def\g{\gamma}
\def\t{\theta}



\begin{thm}\label{thm diff gamma}
     Let $j$ be an integer and $0<j<k$. Define \begin{equation} \nu(c,k)=c^2+(k-c)^2. \end{equation}
     Then $\gamma_k(c)$ is $(\nu(j,k)-2)$-times differentiable at $c=j$. 
\end{thm}

Note that $\nu(c,k)$ reaches its minimum at $c=\lfloor\frac{k+1}{2}\rfloor$, in which case
\begin{equation} \nu \left( \left\lfloor \frac{k+1}{2} \right\rfloor, k \right) =
\left\lfloor\frac{k^2+1}{2}\right\rfloor.\end{equation} Thus, we have
\begin{cor}
 The function $\gamma_k(c)$ is $(\lfloor\frac{k^2+1}{2}\rfloor-2)$-times differentiable for all $0<c<k.$
\end{cor}

The following lemma is essentially proved in Section 4 of~\cite{DHI}.
\begin{lem}\label{lem I_k}
    Let \begin{equation}
         \label{eq:I}
         I_k(u)=\frac{1}{k!}\int_{-\frac{1}{2}}^{\frac{1}{2}}\cdots \int_{-\frac{1}{2}}^{\frac{1}{2}} e^{-2\pi i u\sum_j t_j} \prod_{j<\ell}(t_j-t_\ell)^2 dt_1\cdots dt_k. \end{equation}
    Then \begin{align}
             I_k(u)=\sum_{c=0}^{k} e^{i\pi u (k-2c)} \left(\frac{a(c,k)}{u^{\nu(c,k)}}+O\left(\frac{1}{u^{\nu(c,k)+1}}\right)\right)
    \end{align}
    where 
    \begin{equation} \nu(c,k)=c^2+(k-c)^2 \end{equation}
    and
    \begin{equation}
       a(c,k) =  (-1)^{c}\, (2\pi i)^{-\nu(c,k)}\, G(c+1)^2 \, G(k-c+1)^2 .
    \end{equation}    
\end{lem}
Note that $I_k$ above is essentially the inner multidimensional integral in the expression~\eqref{eq:fourier transform 2} for $\gamma_k$.

\begin{lem}\label{lem not diff}
    We have 
\begin{align}
\gamma_2(c) & =\frac{1}{(2 \pi i)^2}\int_{-\infty}^{\infty} e^{2\pi i u (c-1)} \left( -\frac{1}{u^2}+\frac{\sin(\pi u)^2}{\pi^2 u^4} \right) du \\ \\
& = \left\{
    \begin{array}{ccc} & \dfrac{c^3}{3!}\; ,  & \textrm{ if }\ 0\le c\le 1,\\ \\ & \dfrac{(2-c)^3}{3!}\; ,  & \textrm{ if } \ 1\le c\le 2.
    \end{array}\right.    
\end{align}
    In particular, $\gamma_2(c)$ is not differentiable at $c=1$.
\end{lem}

\begin{proof}[Proof of Theorem~\ref{thm diff gamma}.]
Substituting~\eqref{eq:I} into equation~\eqref{eq:fourier transform 2},
\begin{equation}
    \gamma_k(c)=\frac{1}{G(1+k)^2}\int_{-\infty}^{\infty} e^{2\pi i u (c-\frac{k}{2})} I_k(u) du.
\end{equation}
Moreover, from its multi-integral definition we see that $I_k(u)$ is continuous for all real $u$.
In particular, $I_k(u)$ is bounded near the origin.
Therefore, to prove that $\gamma_k(c)$ is $(\nu(j,k)-2)$-times differentiable at $c=j$, 
it suffices to show that 
\begin{equation} J_{k}(c):=\int_{|u|>1} e^{2\pi i u (c-\frac{k}{2})} I_k(u) du \end{equation}
is $(\nu(j,k)-2)$-times differentiable at $c=j$.

By Lemma~\ref{lem I_k},
\begin{align*}
J_{k}(c) & =\int_{|u|>1}  e^{2\pi i u (c-\frac{k}{2})} \cdot  \sum_{\ell=0}^{k} e^{i\pi u (k-2\ell)} \left(\frac{a(\ell,k)}{u^{\nu(\ell,k)}}+O\left(\frac{1}{u^{\nu(\ell,k)+1}}\right)\right) du\\
&=\sum_{\ell=0}^{k} \int_{|u|>1}  e^{2\pi i u (c-\ell)} \cdot  \left(\frac{a(\ell,k)}{u^{\nu(\ell,k)}}+O\left(\frac{1}{u^{\nu(\ell,k)+1}}\right)\right) du.
\end{align*}
We show that for each $\ell$,
\begin{equation} J_{\ell,k}(c):=\int_{|u|>1}  e^{2\pi i u (c-\ell)} \cdot  \left(\frac{a(\ell,k)}{u^{\nu(\ell,k)}}+O\left(\frac{1}{u^{\nu(\ell,k)+1}}\right)\right) du \end{equation}
is $(\nu(j,k)-2)$-times differentiable at $c=j$.

\textit{Case 1}: $\ell=j$. In this case, we observe that,
for $n=1, 2, \dots, \nu(j,k)-2\;$, the integrals
\begin{align*}
&\int_{|u|>1} \frac{\partial^n}{\partial c^n} \left[ e^{2\pi i u (c-j)} \cdot  \left(\frac{a(j,k)}{u^{\nu(j,k)}}+O\left(\frac{1}{u^{\nu(j,k)+1}}\right)\right)\right]du\\
&=\int_{|u|>1} e^{2\pi i u (c-j)} \cdot (2\pi i u)^n \left(\frac{a(j,k)}{u^{\nu(j,k)}}+O\left(\frac{1}{u^{\nu(j,k)+1}}\right)\right)du\\
&\ll \int_{|u|>1} u^n \left(\frac{a(j,k)}{u^{\nu(j,k)}}+O\left(\frac{1}{u^{\nu(j,k)+1}}\right)\right)du
\end{align*}
are uniformly convergent in $c$. Therefore, $J_{j,k}$ is $(\nu(j,k)-2)$-times differentiable at $c=j$ and, in addition,
\begin{align}\label{eq n-th deriv}
\frac{d^n}{dc^n}J_{j,k}(c)=\int_{|u|>1} e^{2\pi i u (c-j)} \cdot (2\pi i u)^n \left(\frac{a(j,k)}{u^{\nu(j,k)}}+O\left(\frac{1}{u^{\nu(j,k)+1}}\right)\right)du
\end{align}
for $n=1, 2, \dots, \nu(j,k)-2\,$.

\textit{Case 2}: $\ell\ne j$. In this case, we show that $J_{\ell,k}(c)$ is in fact $C^{\infty}$ at $c=j.$ To prove this, it suffices to show that
\begin{equation} \int_{|c|>1}  e^{2\pi i u \delta} \,\frac{du}{u}\end{equation} is $C^{\infty}$ at $\delta\ne 0$.

Using integration by parts repeatedly we see that
\begin{equation} \int_{|c|>1}  e^{2\pi i u \delta} \,\frac{du}{u}=\frac{m!}{(2\pi i\delta)^m}\int_{|c|>1} e^{2\pi i u \delta} \,\frac{du}{u^{m+1}}+O_m(\delta^{-1}+\delta^{-m}) \end{equation}
for any $m\in\mathbb N$ and real $\delta\ne 0$, where the Big-$O$ term is a $C^{\infty}$ function for $\delta\ne 0$. Also, by uniform convergence (see a similar argument in Case 1)
\begin{equation} \frac{m!}{(2\pi i\delta)^m}\int_{|c|>1} e^{2\pi i u \delta} \,\frac{du}{u^{m+1}} \end{equation}
is $(m-1)$-times differentiable at $\delta\ne 0$. It follows that
\begin{equation} \int_{|c|>1}  e^{2\pi i u \delta} \,\frac{du}{u} \end{equation} is $(m-1)$-times differentiable at $\delta\ne 0$. Since $m$ is arbitrary, we have
\begin{equation} \int_{|c|>1}  e^{2\pi i u \delta} \,\frac{du}{u} \end{equation} is $C^{\infty}$ at $\delta\ne 0$.

Combining Case 1 and Case 2 we obtain that
\begin{equation} J_{k}(c):=\int_{|u|>1} e^{2\pi i u (c-\frac{k}{2})} I_k(u) du \end{equation}
is $(\nu(j,k)-2)$-times differentiable at $c=j$, and therefore, so is $\gamma_k(c).$

\medskip

Lastly, we show that 
\begin{equation} \left(\dfrac{d}{dc}\right)^{\nu(j,k)-2}\gamma_k(c) \end{equation}
is not differentiable at $c=j$.
It suffices to show that
\begin{equation} \left(\dfrac{d}{dc}\right)^{\nu(j,k)-2}J_{j,k} \end{equation}
is not differentiable at $c=j$. By equation~\eqref{eq n-th deriv} we have
\begin{align*} \left(\dfrac{d}{dc}\right)^{\nu(j,k)-2}J_{j,k} &  = \int_{|u|>1} e^{2\pi i u (c-j)} \cdot (2\pi i u)^{\nu(j,k)-2}\\ & \qquad \qquad \left(\frac{a(j,k)}{u^{\nu(j,k)}}+O\left(\frac{1}{u^{\nu(j,k)+1}}\right)\right)du.\end{align*}
Again, by the uniform convergence argument we see that
\begin{align*} \int_{|u|>1} e^{2\pi i u (c-j)} \cdot (2\pi i u)^{\nu(j,k)-2} \cdot O\left(\frac{1}{u^{\nu(j,k)+1}}\right)du\end{align*}
is differentiable at $c=j$. 
Therefore, it remains to show that 
\begin{align*} \int_{|u|>1} e^{2\pi i u (c-j)} \cdot (2\pi i u)^{\nu(j,k)-2} \cdot \frac{a(j,k)}{u^{\nu(j,k)}}du\end{align*}
is not differentiable at $c=j$, or equivalently,
\begin{align*} \int_{|u|>1} e^{2\pi i u (c-1)} \cdot \frac{du}{u^2}\end{align*}
is not differentiable at $c=1$. 

It follows from Lemma~\ref{lem not diff} that
\begin{align*}
\int_{|u|>1} e^{2\pi i u (c-1)} \left( -\frac{1}{u^2}+\frac{\sin(\pi u)^2}{\pi^2 u^4} \right) du
\end{align*}
is not differentiable at $c=1$. 
Since
\begin{align*}
\int_{|u|>1} e^{2\pi i u (c-1)} \cdot \frac{\sin(\pi u)^2}{\pi^2 u^4} du
\end{align*}
is differentiable at $c=1$, we see that 
\begin{align*} \int_{|u|>1} e^{2\pi i u (c-1)} \cdot \frac{du}{u^2}\end{align*}
is not differentiable at $c=1$. 
This ends our proof of Theorem~\ref{thm diff gamma}.

\end{proof}

The highly smooth nature of $\gamma_k(c)$ was first observed empirically by Conrey in the related problem of
determining the asymptotics of the second moment of Dirichlet polynomials whose coefficients are
$k$-th divisor numbers. Specifically, he defines
\begin{eqnarray*}
M_k(c)=\lim_{T\to \infty}\frac{(k^2)!}{a_k T (\log T)^{k^2}}
\int_0^T \left|\sum_{n=1}^N \frac{d_k(n)}{n^{1/2+it}}\right|^2 ~dt
\end{eqnarray*}
for integer values of $k$ and $N=T^c$ with $c>0$, and determined $M_k(c)$ for $k\leq 4$ (conjecturally for $k=3,4$).
By comparing Conrey's tables (personal communication) for $M_k(c)$ with our tables for
$\gamma_k(c)$, it appears to be the case that the derivative of $M_k(c)$ is equal to $(k^2)! \gamma_k(c)$.
Bettin~\cite{B} has proven the analogous smoothness for the polynomials $M_k(c)$.

%


\section{Expansion for $\log D_k(t)$ and the limiting behaviour of $\gamma_k(c)$}

Notice that 
\begin{equation}
    g^{(n)}(0) = \int_0^1 (-x)^n dx = (-1)^n/(n+1).
\end{equation}
Thus, pulling out powers of $-1$ from the determinant, of which there are an
even number, we have
$D_k(0) = \det_{k\times k} (1/(i+j-1))$, which is a special case of
the Cauchy determinant and thus
\begin{equation}
 D_k(0) = G(k+1)^4/G(2k+1).
\end{equation}

Now, $D_k(t)$ satisfies the Toda equation~\cite{S}:
\begin{equation}
    \frac{D_{k-1}(t) D_{k+1}(t)}{D_k(t)^2} = \frac{D^{''}_k(t)}{D_k(t)}-\frac{(D_k'(t))^2}{D_k(t)^2 }
    =(\log(D_k(t)))''
    \label{eq:lewis carroll}
\end{equation}
This follows from a recursion of Dodgson (aka Lewis Carroll) for computing determinants~\cite{D}.
Define $c_m(k)$ by:
\begin{equation}
    \label{eq:log D_k series}
    D_k(t) = D_k(0) \exp\left(\sum_1^\infty \frac{c_m(k)}{m} t^m\right).
\end{equation}
Take the log derivative of the lhs and rhs of the above identity, substitute
the series for $\log(D_k(t))$, and clear the denominator of the rhs.
Comparing coefficients gives the recursion, for $M>2$:
\begin{multline}
    c_M(k) = \frac{1}{(M-1)(M-2)} \sum_{m=0}^{M-3} (m+1) c_{m+2}(k) \\
    \times ( c_{M-m-2}(k-1) + c_{M-m-2}(k+1) - 2 c_{M-m-2}(k) )
\end{multline}
This recursion determines the coefficients $c_M(k)$ in terms of $c_1(k), \ldots, c_{M-2}(k)$.

To get $c_1(k)$:
\begin{equation}
    c_1(k) = D'_k(0)/D_k(0).
\end{equation}
One can differentiate $D_k(t)$ by using the product rule to get a sum of
determinants where we differentiate the $i$-th row. However, because the entries
of $D_k(t)$ are derivatives, differentiating the $i$-th row produces a row that
matches the one below it, and the determinant vanishes. Thus, only the last of
these terms, where we differentiate the last row, survives. However, that
determinant is also a Cauchy determinant with $i,j$ entry $(-1)^{i+j-1}/(i+j-1)$ as
before, except for the last row where the entry is $(-1)^{i+j}/(i+j)$.

Using the formula for the Cauchy determinant, a lot of cancellation occurs and
we get
\begin{equation}
    c_1(k) = -k/2.
\end{equation}

To determine $c_2(k)$, substitute $t=0$ into identity~\eqref{eq:lewis carroll}.
On the lhs:
\begin{multline}
    D_{k-1}(0) D_{k+1}(0) /D_k(0)^2\\
    = G(k)^4 G(k+2)^4 G(2k+1)^2 / (G(2k-1) G(2k+3) G(k+1)^8) \\
    = k^2/(4 (4k^2-1)).
\end{multline}
On the rhs, the constant term of $(\log(D_k(t)))''$ is $c_2(k)$, so
\begin{equation}
    c_2(k) = k^2/(4 (4k^2-1)).
\end{equation}

The recursion, along with the initial two terms determine all the $c_m(k)$'s.
For example, $c_3(k)=0$, and
\begin{equation}
    c_4(k) = \frac {{k}^{2}}{ 16\left( 4k^2-1 \right) ^{2} \left( 4k^2-9 \right)  }.
\end{equation}

We can apply the above to determine the asymptotic expansion of $\gamma_k(c)$ in a large
neighbourhood of $k/2$. To do so, isolate the $m=1,2$ terms from the series~\eqref{eq:log D_k series},
substitute into~\eqref{eq:gamma as ft D_k} with $t=2\pi i u$, and compose the series for $\exp$ with
that of the terms $m\geq 3$ of ~\eqref{eq:log D_k series}, to get that the integrand of~\eqref{eq:gamma as ft D_k}
equals:
\begin{equation}
    \label{eq:integrand isolate quadratic exponential}
    \exp\left(-\frac{(k\pi u)^2}{2(4 k^2 -1)} + 2\pi i (c-k/2) u \right)
    \left(
        1 +
        \frac {k^2 (\pi u)^4}
        { 4\left( 4k^2-1 \right) ^{2} \left( 4k^2-9 \right)  }
        + \ldots
    \right).
\end{equation}
One can obtain more terms, if desired, from the recursion for $c_M(k)$. 
We thus have the following asymptotic expansion:
\begin{thm}

Let $b_k = 8(1-1/(4k^2))$ and $c=k/2+o(k)$.
Then
\begin{multline}
    \gamma_k(c) \sim \frac{G(k+1)^2}{G(2k+1)}
    \sqrt{ \frac{b_k}{\pi} } \exp( - b_k (c-k/2)^2 )\\
    \times \Bigg(
    1+\frac{1}{4k^2-9}
        \bigg( \frac{64(c-k/2)^4-24(c-k/2)^2+3/4}{k^2}\\
        -2\frac{(c-k/2)^2(16(c-k/2)^2-3)}{k^4}
        +4\frac{(c-k/2)^4}{k^6} \bigg)
        +\ldots
        \Bigg).
\end{multline}
i.e. Gaussian near the centre.
\end{thm}

\section{Elliptic aliquot cycles}

The basic method used to pass from~\eqref{eq:def of gamma} to equation~\eqref{eq:fourier transform 1}
can be used in the context of elliptic aliquot cycles.

Let $\bs p=(p_1,\dots,p_d)$ be a $d$-tuple of distinct primes.
Let $\alpha(\bs p)$ be the
probability of choosing random
and independently $d$
elliptic curves $E_1,\dots,E_d$ over $\F_{p_1},\dots,\F_{p_d}$, respectively,
with the property that $|E(\F_{p_j})|=p_{j+1}$, for $j\in\{1,\dots,d\}$.
Here, $p_{d+1}=p_1$. We are choosing the curves $E_j$ uniformly from the
set of isomorphism classes of elliptic curves over $\F_p$.

David, Koukoulopoulos, and Smith~\cite{DKS}  gave an asymptotic for the average of $\alpha(\bs p)$
over the set
\begin{equation}
    \CP_d(x) =
    \{(p_1,\dots,p_d): p_1\le x\}.
\end{equation}
(Hasse's bound implies that $\alpha(\bs p)=0$ unless $|p_{j+1}-p_j-1|<2\sqrt{p_j}$ for $1\leq j \leq d$).


\begin{thm}[DKS]\label{aliquot}
For any fixed $A>0$,
{\small
\[
\sum_{\bs p \in \CP_d(x)} \alpha(\bs p)
    = C_{\text{aliquot}}^{(d)} \int_2^x \frac{du}{2\sqrt{u}(\log u)^d}
        + O_A\left( \frac{\sqrt{x}}{(\log x)^A} \right)
     \sim C_{\text{aliquot}}^{(d)} \frac{\sqrt{x}}{(\log x)^d} ,
\]
where{\footnotesize
\[
C_{\text{aliquot}}^{(d)} := I_{\text{aliquot}}^{(d)} \cdot
    \prod_{\ell}\frac{\ell^d\cdot
        \# \left\{ \bs\sigma\in \GL_2(\Z/\ell\Z)^d :
        \begin{array}{l}
        \det(\sigma_j) +1 -\tr(\sigma_j) \equiv  \det(\sigma_{j+1}) (\ell)\\
        \mbox{for $1\le j\le d$, where $\sigma_{d+1}=\sigma_1$}
        \end{array} \right\}}{\left| \GL_2(\Z/\ell \Z) \right|^d}
\]
}
with
\[
I_{\text{aliquot}}^{(d)} := \frac{2^d}{\pi^d}
    \idotsint\limits_{\substack{|t_j|\le 1\ (1\le j\le d-1) \\ |t_1+\cdots+t_{d-1}|\le 1}}
        \sqrt{1-(t_1+\cdots+t_{d-1})^2} \prod_{j=1}^{d-1}\sqrt{1-t_j^2} \ d t_1\cdots d t_{d-1} .
\]
}
\end{thm}

Let
\begin{equation}
    \label{eq:I(d)}
    I(d):=
    \idotsint\limits_{\substack{|t_j|\le 1\ (1\le j\le d-1) \\ |t_1+\cdots+t_{d-1}|\le 1}}
    \sqrt{1-(t_1+\cdots+t_{d-1})^2} \prod_{j=1}^{d-1}\sqrt{1-t_j^2} \ d t_1\cdots d t_{d-1}.
\end{equation}
$I(1)=1$, $I(2)=4/3$. One might wonder if $I(d)$ persists in being rational. We will
show, for $d=3$, that this seems unlikely.

Replacing the Dirac delta function by the integral in~\eqref{eq:delta}, we have
\begin{equation}
    \label{eq:I(d) 3}
    I(d) =
    \int_{-\infty}^{\infty}
    \int_{[-1,1]^d} \prod_1^d (1-t_j^2)^{1/2}
    \exp\left(2\pi i y \sum t_j\right)
    d t_1\cdots d t_d d y
\end{equation}
But
\begin{equation}
    \int_{-1}^{1} (1-t^2)^{1/2} \exp(2\pi i y t) d t = J_1(2\pi y)/(2y),
\end{equation}
($J$-Bessel function on the rhs). Separating the integral, we get
\begin{equation}
    \label{eq:I(d) 4}
    I(d) =
    \int_{-\infty}^{\infty} \left( \frac{J_1(2\pi y)}{(2y)} \right)^d dy,
\end{equation}
i.e. a one dimensional integral.

This formula can be used to efficiently evaluate $I(d)$ for, say, $d=3,4,\ldots$, for example
with Poisson summation.

Let $f \in L^1({\mathbb{R}})$ and let
\begin{equation}
    \hat{f}(y) = \int_{-\infty}^\infty f(t) e^{-2\pi i y t} dt.
\end{equation}
denote its Fourier transform. 
The Poisson summation formula asserts, for, say, $f$ continuous,
that
\begin{equation}
    \sum_{n=-\infty}^{\infty} f(n) = \sum_{n=-\infty}^{\infty} \hat{f}(n)
\end{equation}
provided the rhs converges absolutely and that $\sum f(n+v)$ converges uniformly
in $v$ on compact sets.

Let $\Delta >0$. By a change of variable
\begin{equation}
    \Delta \sum_{n=-\infty}^{\infty} f(n\Delta) = \sum_{n=-\infty}^{\infty} \hat{f}(n/\Delta)
    = \hat{f}(0) + \sum_{n \neq 0} \hat{f}(n/\Delta),
\end{equation}
so that
\begin{equation}
    \int_{-\infty}^\infty f(t) d t
    -\Delta \sum_{n=-\infty}^{\infty} f(n\Delta)
    =-\sum_{n \neq 0} \hat{f}(n/\Delta)
\end{equation}
tells us how closely the Riemann sum $\Delta \sum_{n=-\infty}^{\infty} f(n\Delta)$
approximates the integral $\int_{-\infty}^\infty f(t)dt$.

Apply, with
\begin{equation}
    f(y)= \left( \frac{J_1(2\pi y)}{(2y)} \right)^d.
\end{equation}

Note that
\begin{equation}
    \int_{-\infty}^{\infty} \frac{J_1(2\pi y)}{(2y)} \exp(-2\pi i u x) d x
    =
    \begin{cases}
        (1-u^2)^{1/2}, \qquad |u| \leq 1, \\
        0, \qquad \text{otherwise}.
    \end{cases}
\end{equation}
Therefore, the Fourier transform of $ \left( \frac{J_1(2\pi y)}{(2y)} \right)^d$, being
the $d$-fold convolution of $(1-u^2)^{1/2}$ with itself, is supported in $|u|\leq d$.

Hence, in the Poisson sum method, any choice of $\Delta \geq 1/d$ gives {\it no} remainder
in the Poisson formula (i.e. 0 contribution from them $|n| \geq 1 $ terms).
Thus, taking $\Delta = 1/d$ gives:
\begin{equation}
    \label{eq:I(d) 5}
    I(d) =
    \int_{-\infty}^{\infty} \left( \frac{J_1(2\pi y)}{(2y)} \right)^d dy
    = \frac{1}{d} \sum_{-\infty}^{\infty} \left( \frac{J_1(2\pi n/d )}{(2n/d)} \right)^d.
\end{equation}
Furthermore, $J_1(z) \sim \sqrt{\frac{2}{\pi z}} \cos(z-3\pi/4)$, hence the sum on the right
has terms that are $\ll (2\pi)^{-d} (n/d)^{-3d/2}$. Thus with $d=3$, the first  million terms of the
sum gives more than twenty digits accuracy.

One can accelerate the convergence of the sum further using the asymptotics of the $J$-Bessel function,
and algorithms for the evaluation of the polylogarithm $\text{Li}_s(z)=\sum_1^\infty z^n/n^s$.
Or one can cheat and just use a blackbox like Maple to evaluate~\eqref{eq:I(d) 4}, with $d=3$:
\begin{multline}
    I(3)=
    1.7053570421915038354985956872898996791331386909\\
    7890590667136169819331192007797559594679011\ldots
\end{multline}
Let $A_n/B_n$ be the $n-th$ convergent of the continued fraction of the real number $\alpha$.
If $p,q \in \Z$ satisfies:
\begin{equation}
    |\alpha-p/q| < |\alpha-A_n/B_n|
\end{equation}
then $q>B_n$. Therefore, computing the continued fraction for $I(3)$,
the 85-th convergent is:
\begin{equation}
    {\frac {14703927951211792459205597491632973549428444428}{
    8622199098152613288048825699460716423721576467}}
\end{equation}
(and $|I(3)-A_{85}/B_{85}| \neq 0$.
With given precision, there is a limit to how many convergents we can meaningfully use).

Thus, if $I(3)$ is rational, then it has denominator at least $10^{45}$. It would not be too
difficult to increase the denominator to hundreds or thousands of digits (millions of digits with some
effort), assuming $I(3)$ is irrational.

Maple's identify command did not turn up any obvious expressions for $I(3)$ in terms
of algebraic numbers and known constants.

One can also determine the behaviour of $I(d)$ for large $d$. Writing
\begin{equation}
    \left( \frac{J_1(2\pi y)}{(2y)} \right)^d
     = \left(\frac{\pi}{2} \right)^d \exp\left( d \log(J_1(2\pi y)/(\pi y))\right),
    \label{eq:J1 as exp log}
\end{equation}
expanding $J_1$ in its Maclaurin series, and pulling out the $y^2$ term, the above becomes
\begin{multline}
    \left(\frac{\pi}{2} \right)^d \exp\left(-\frac{d \pi^2 y^2}{2} \right) \\
    \times \exp\left(-\frac{d \pi^4 y^4}{24} -\frac{d\pi^6 y^6}{144}
    -\frac{d\pi^8 y^8}{720} - \frac{13 d \pi^{10} y^{10}}{43200} + \ldots\right).
    \label{eq:J1 exp log series}
\end{multline}
    Taking the Maclaurin series of the latter exponential (truncated with remainder term),
    we thus get the asymptotic expansion
    \begin{multline}
        \label{eq:I(d) 6}
        I(d) =
        \int_{-\infty}^{\infty} \left( \frac{J_1(2\pi y)}{(2y)} \right)^d dy \\
        \sim \left(\frac{\pi}{2}\right)^{d-1/2} \frac{1}{d^{1/2}}
        \left(
            1 - \frac{1}{8d}-\frac{5}{384d^2}+\frac{7}{3072d^3}+\frac{3829}{491520d^4}+\ldots 
        \right).
    \end{multline}
\newpage

\begin{table}[H]
\small
\begin{tabular}{|c|c|l|}
\hline
$k$ & $ j$ & $ (k^2-1)!\gamma_k(c)$ \\
\hline
$2$ & $ 0$ & $ c^3$ \\
\hline
& $ 1$ & $ (2-c)^3$ \\
\hline 
\hline
$3$ & $ 0$ & $ c^8$ \\
\hline
& $ 1$ & $ -2 c^{8}+24 c^{7}-252 c^{6}+1512 c^{5}-4830 c^{4}$ \\
& & $+8568 c^{3}-8484 c^{2}+4392 c-927$ \\
\hline
& $ 2$ & $ (c-3)^8$ \\
\hline
\hline
$4$ & $ 0$ & $ c^{15}$ \\
\hline
& $ 1$ & $ -3 c^{15}+60 c^{14}-1680 c^{13}+29120 c^{12}-294840 c^{11}+1873872 c^{10}-7927920 c^{9}$\\
&&$+23268960c^{8}-48674340 c^{7}+ 73653580 c^{6}-80912832 c^{5}+63969360 c^{4}$\\
&&$-35497280 c^{3}+13131720 c^{2}-2910240 c+292464$ \\
\hline
& $ 2$ & $ 3 c^{15}-120 c^{14}+3360 c^{13}-58240 c^{12}+644280 c^{11}-4948944 c^{10}+28428400 c^{9}$\\
&&$-128700000 c^{8}+470398500 c^{7}-1381480100 c^{6}+3179336160 c^{5}-5531176560 c^{4}$\\
&&$+6950332480 c^{3}-5910494520 c^{2}+3031004640 c-705916304$\\
\hline
& $ 3$ & $ (4-c)^{15}$\\
\hline
\hline
$5$ & $ 0$ & $ c^{24}$ \\
\hline
& $ 1$ & $ -4 c^{{24}}+120 c^{{23}}-6900 c^{{22}}+253000 c^{{21}}-5578650 c^{{20}}+79695000 c^{19}$\\
&&$-785367660 c^{18}+5598232200 c^{17}- 29915282925 c^{16}+123134189200 c^{15}$\\
&&$-398517412920 c^{14}+1029946456560 c^{13}-2149736416100 c^{12}+3651921075600 c^{11}$\\
&&$-5072249298600 c^{10}+5768661885360 c^{9}-5363308269495 c^{8}+4055447662200 c^{7}$\\
&&$-2470634081300 c^{6}+1194550480200 c^{5}-447845361810 c^{4}+125530048600 c^{3}$\\
&&$-24758793900 c^{2}+3065085000 c-179192775$\\
\hline
& $ 2$ & $ 6 c^{24}-360 c^{23}+20700 c^{22}-759000 c^{21}+17798550 c^{20}-292215000 c^{19}$\\
&&$+3673797820 c^{18}-38235839400 c^{17}+347123925225 c^{16}-2790376974000 c^{15}$\\
&&$+19589544660840 c^{14}-117507788504400 c^{13}+592028782736300 c^{12}$\\
&&$-2479096272534000 c^{11}+8573537591434200 c^{10}-24367026171730000 c^{9}$\\
&&$+56603181050415945 c^{8}-106665764409131400 c^{7}+161304132700472300 c^{6}$\\
&&$-192656070655587000 c^{5}+177464649282553710 c^{4}-121528934511474600 c^{3}$\\
&&$+58223870087874900 c^{2}-17407730744067000 c+2443806916000825$\\
\hline
& $ 3$ & $ -4 c^{24}+360 c^{23}-20700 c^{22}+759000 c^{21}-18861150 c^{20}+345345000 c^{19}$\\
&&$-4991492660 c^{18}+59676982200 c^{17}-604502001675 c^{16}+5220961534800 c^{15}$\\
&&$-38343917872920 c^{14}+238359873297840 c^{13}-1250073382257700 c^{12}$\\
&&$+5522495132708400 c^{11}-20539021982760600 c^{10}+64263112978594640 c^{9}$\\
&&$-168820549421134545 c^{8}+370693368908418600 c^{7}-674525363862958300 c^{6}$\\
&&$+1002229415508043800 c^{5}-1187187920423969310 c^{4}+1078975874367012600 c^{3}$\\
&&$-706068990841773900 c^{2}+295689680026989000 c-59394510856327775$\\
\hline
& $ 4$ & $ (5-c)^{24}$ \\
\hline
\end{tabular}
\caption
{The polynomials $(k^2-1)!\gamma_k(c)$ for $k\leq 5$ and $j \leq c \leq j+1$.}
\label{table:gamma_k}
\end{table}
 
\begin{table}[H]
\tiny
\begin{tabular}{|c|c|l|}
\hline
$k$ & $ j$ & $ (k^2-1)!\gamma_k(c)$ \\
\hline
\hline
$6$ & $ 0$ & $ c^{35}$ \\
\hline
$6$ & $ 1$ & $ -5 c^{35}+210 c^{34}-21420 c^{33}+1413720 c^{32}-56862960 c^{31}+1501747632 c^{30}$\\
&&$-27736558080 c^{29}+375954464160 c^{28}-3881009646360 c^{27}+31410293440680 c^{26}$\\
&&$-203947162827408 c^{25}+1082230579684800 c^{24}-4764220775823600 c^{23}$\\
&&$+17613096754503600 c^{22}-55229306110228800 c^{21}+148080133608311520 c^{20}$\\
&&$-341689133815514100 c^{19}+682008750903872700 c^{18}-1182119446613536200 c^{17}$\\
&&$+1784232273468783600 c^{16}-2349159980084905680 c^{15}+2699953776702032400 c^{14}$\\
&&$-2707997790067516800 c^{13}+2366932574161864800 c^{12}-1798264701411305400 c^{11}$\\
&&$+ 1182907170763213896 c^{10}-670007069282572560 c^{9}+324322366699605120 c^{8}$\\
&&$-132818300667235920 c^{7}+45395326648924560 c^{6}-12709759385961792 c^{5}$\\
&&$+2839179794146080 c^{4}-486611119673910 c^{3}+60083734292610 c^{2}$\\
&&$-4757721939180 c+181451828088$\\
\hline
& $ 2$ & $ 10 c^{35}-840 c^{34}+85680 c^{33}-5654880 c^{32}+238447440 c^{31}-7029581328 c^{30}$\\
&&$+158939827200 c^{29}-3010298623200 c^{28}+51174168784200 c^{27}-802885194480600 c^{26}$\\
&&$+11485501718811120 c^{25}-145954772087342400 c^{24}+ 1615205663712622800 c^{23}$\\
&&$-15414821245929142800 c^{22}+126507768912420350400 c^{21}-893399034384858022560 c^{20}$\\
&&$+5440022414523749814300 c^{19}-28627456041998656712100 c^{18}+130462364245768533732600 c^{17}$\\
&&$-515683796529615245254800 c^{16}+1769595318452023551221040 c^{15}-5272695333575690900655600 c^{14}$\\
&&$+13632520546818627517123200 c^{13}-30536223709478278133815200 c^{12}+59100950810144250579990600 c^{11}$\\
&&$- 98447935269887910573290424 c^{10}+140369638227928515300288240 c^{9}-170046927222112798851396480 c^{8}$\\
&&$+173284197564689124463669680 c^{7}-146552294343347207749027440 c^{6}+100980418141793007531096768 c^{5}$\\
&&$-55222971916535322127277280 c^{4}+23052485974924851589246410 c^{3}-6898544814307888233994110 c^{2}$\\
&&$+1317633501288006725436180 c-120657836168926671721608$\\
\hline
\end{tabular}
\caption
{$(k^2-1)!\gamma_k(c)$ for $k=6$ and $j \leq c \leq j+1$, $j=0,1,2$. The polynomials
for $j=3,4,5$ can be determined from the above using $\gamma_k(c) = \gamma_k(k-c)$.}
\label{table:gamma_k b}
\end{table}

\section{Acknowledgments}

The third author is supported in part by an NSERC Discovery Grant. The first and third authors conducted part of this research at the American Institute of Mathematics. We thank AIM for support, and Brian Conrey and Sandro Bettin or helpful comments.

%
%
%
%
%
%


\end{document}